\documentclass{amsart}
\usepackage{graphicx} 
\usepackage{amsmath,amsfonts,amssymb,amsthm}
\usepackage[english]{babel}
\usepackage{tikz}
\tikzset{
  dots/.style={line width=1pt, line cap=round, dash pattern=on 0pt off 5\pgflinewidth}
}

\usepackage[backend=biber]{biblatex}
\usepackage{tkz-tab}

\bibliography{biblio}

\newtheorem{thm}{Theorem}
\newtheorem{rem}[thm]{Remark}

\newtheorem{lem}[thm]{Lemme}

\newtheorem{prop}[thm]{Proposition}

\usepackage{graphicx} 

\title{Almost-invariant half-spaces of operators on $\omega$}
\author{Noémie Fougnies}

\address{Noémie Fougnies, Département de Mathématique, Université de Mons, 20 Place du Parc, 7000 Mons, BELGIUM}
\email{noemie.fougnies@umons.ac.be}

\date{}

\thanks{The author was partially supported by the F.R.S-FNRS through funding PDR No. T.0285.26.}

\begin{document}

\begin{abstract}
We study the Almost-Invariant Subspace Problem for operators defined on non-normable Fréchet spaces and especially on $\omega$.
This problem is stated as follows: "Given a bounded operator $T$ on an infinite-dimensional complex Fréchet space $X$, can we always find a closed almost-invariant half-space?". 
In this paper we solve the problem for the space $\omega$ by showing that an operator $T$ on $\omega$ possesses a closed almost-invariant half-space if and only if $T$ is conjugate to $F + R + \lambda \operatorname{Id}$ or to $B + R + \lambda \operatorname{Id}$ where $F$ is the Forward shift and $B$ is the Backward shift and where $R$ is a finite rank operator and $\lambda \in \mathbb{C}$. 
We also investigate the case of operators defined on $X \oplus \omega$ where $X$ is an infinite-dimensional Fréchet space with a continuous norm and we show that in this case, every operator defined on $X \oplus \omega$ possesses a closed almost-invariant half-space.
\end{abstract}

\maketitle

\section{Introduction}

The Invariant Subspace Problem is a well-known problem in operator theory.
Given $X$ a Banach space, we are interested to know if an operator $T \in L(X)$ possesses a non-trivial closed invariant subspace, where a closed subspace $M$ is said to be a non-trivial invariant subspace of $T$ if $T(M) \subseteq M$ and  $M \neq \{0\}$ and $X$.
This problem was first addressed for operators defined on Hilbert spaces.
Indeed, the problem seems to appear after Beurling characterizes all the closed invariant subspaces for a certain operator defined on a Hilbert space \cite{beurling1949two}.
The problem also begins with the unpublished work of von Neumann, who in 1949 proved that every compact operator defined on a Hilbert space has a non-trivial invariant subspace. 
Thanks to Lomonosov and his famous theorem proved in 1973, we even have that every operator commuting with a compact operator has a non-trivial closed invariant subspace \cite{lomonosov1973invariant}. 
Moreover, another significant breakthrough was demonstrated in 2011 by Argyros and Haydon who showed that there exists a separable Banach space of infinite dimension such that every bounded operator defined on this space has a closed invariant subspace \cite{argyros2011hereditarily}. 

On the other hand, we also know that there exist some counterexamples of this problem.
Indeed, Enflo and Read demonstrated that there exists an operator on separable infinite-dimensional Banach spaces that does not possess a non-trivial closed invariant subspace \cite{enflo1987invariant, read1984solution}. 
A classical counterexample, shown by Read, is that there exists an operator on $\ell^1$ without any non-trivial closed invariant subspace \cite{read1985solution}. 

We can now ask ourselves, for example in the case of $\ell^1$, how far some operator is from having a closed non-trivial invariant subspace.
Given a bounded operator $T$ on an infinite-dimensional Banach space $X$, can we always find a finite rank perturbation of it that possesses a closed non-trivial invariant subspace?
If $M$ is an invariant subspace under the action of a finite rank perturbation of $T$, we will say that $M$ is almost-invariant for $T$.
However, each subspace of $X$ that is finite-dimensional or of finite codimension is always almost-invariant for $T$.
For this reason, we may wonder if, given an operator $T$, there always exists a closed almost-invariant subspace of infinite dimension and infinite codimension. 
A subspace of infinite dimension and infinite codimension is called a half-space.
So we get the Almost-Invariant Half-space Problem: "Given a bounded operator $T$ on an infinite-dimensional complex Banach space $X$, can we always find a closed almost-invariant half-space?". 
This problem was first introduced by Androulakis, Popov, Tcaciuc and Troitsky in 2009 \cite{androulakis2009almost} and was completely solved in 2019 by Tcaciuc who proved that every bounded operator on an infinite-dimensional complex Banach space admits a closed almost-invariant half-space \cite{tcaciuc2019invariant}.

Many topological vector spaces, which play an important role in functional analysis, have a topology that cannot be defined from a norm but rather by a family of seminorms.
This is the case, for example, for the vector space of holomorphic functions $H(\mathbb{C})$ or for $\omega = \mathbb{C}^{\mathbb{Z}_+}$, the space of all complex sequences where $\mathbb{Z}_{+}$ is the set of non-negative integers.
These spaces are actually Fréchet spaces. 
Fréchet spaces generalize Banach spaces by having a topology given by an increasing and separating sequence $(p_j)_{j \geqslant 0}$ of seminorms.

The Invariant Subspace Problem has also been investigated in the context of Fréchet spaces. We know for instance that the space of holomorphic functions $H(\mathbb{C})$ admits an operator without non-trivial closed invariant subspace \cite{golinski2013operators}, while every operator on $\omega$ possesses a non-trivial closed invariant subspace \cite{korber2013die}.

The goal of this paper is to investigate the Almost-Invariant Subspace Problem for non-normable Fréchet spaces. 
In this paper, we will focus on Fréchet spaces endowed with an increasing and separating sequence of seminorms $(p_j)_{j\geqslant 0}$ such that for every $j\geqslant 0$, $\ker p_j$ is an infinite-dimensional subspace and $\ker p_{j+1}$ is of finite codimension in $\ker p_j$. In other words, we will look at Fréchet spaces isomorphic to $X \oplus \omega$ where $X$ is a Fréchet space with a continuous norm and where the space $\omega$ is endowed with the sequence of seminorms $(p_j)_{j\geqslant 0}$ given by $p_j(x) = \max_{n \leqslant j} |x_n|$ for $j \geqslant 0$ and $x \in \omega$.


While the answer to the Almost-Invariant Subspace Problem is positive for any bounded operator on any infinite-dimensional complex Banach space, it will not be the case for any bounded operator on any infinite-dimensional complex Fréchet space.
Indeed, we will show that the Forward shift $F$ and the Backward shift $B$ does not possess a closed almost-invariant half-space on $\omega$. We recall that these operator are defined by
$$
F(x_0,x_1,x_2,\ldots) = (0,x_0,x_1,\ldots) \quad \text{ and } \quad B(x_0,x_1,x_2,\ldots) = (x_1,x_2,x_3,\ldots).
$$
Recall that every operator on $\omega$ possesses a non-trivial closed invariant subspace. Unlike the case of operators on any infinite-dimensional complex Banach spaces, it is thus possible on some Fréchet spaces to find operators with non-trivial closed invariant subspaces but no closed almost-invariant half-spaces.

In fact, the existence of non-trivial closed invariant subspaces also holds for any bounded operator on $X\oplus \omega$ where $X$ is a Fréchet space with a continuous norm~\cite{menet2018invariant}, while the existence of closed almost-invariant half-spaces will depend on the dimension of $X$.

\begin{thm}\label{thmX}
Let $X$ be a Fréchet space with a continuous norm. 
Then, each operator defined on $X \oplus \omega$ possesses a closed almost-invariant half-space if and only if $X$ is infinite-dimensional.
\end{thm}

This result will easily follow from our study of $\omega$ and will be proved in the last section of this paper.

Given the absence of closed almost-invariant half-spaces for the shifts $F$ and $B$ on $\omega$, one wonders if there are other examples of operators on $\omega$ that do not admit closed almost-invariant half-spaces. We can actually fully characterize these operators, and this leads us to the main result of this paper, which can be stated as follows:


\begin{thm} \label{mainthm}
Let $T \in L(\omega)$. 
The following statements are equivalent.
\begin{enumerate}
    \item $T$ does not possess a closed almost-invariant half-space;
    \item $T$ does not possess a closed invariant half-space;
    \item There exist $\lambda \in \mathbb{C}$ and an operator $R \in L(\omega)$ of finite rank such that $T$ is conjugate to $B+ \lambda \operatorname{Id}  + R$ or to $F+ \lambda \operatorname{Id}+ R$,
\end{enumerate}
where $F$ is the Forward shift and $B$ is the Backward shift.
\end{thm}
We recall that an operator $T$ is said to be conjugate to an operator $S$ if there exists an invertible operator $P$ such that $T = P^{-1}SP$.

We will in fact first prove that $T$ does not possess a closed invariant half-space is equivalent to having that $T$ is conjugate to $B+ \lambda \operatorname{Id}  + R$ or to $F+ \lambda \operatorname{Id}+ R$ for $R$ a finite rank operator and for $\lambda \in \mathbb{C}$.
The equivalence between having a closed almost-invariant half-space and having a closed invariant half-space will then follow from this characterization.\\

This paper is organized as follows. In Section~\ref{Sec2}, we show for which types of transformations the property of possessing an invariant half-space is preserved. In Section~\ref{SecFB}, we study the case of the Forward shift and the Backward shift and we show that these operators do not possess a closed almost-invariant half-space.
Actually, we will even demonstrate one of the implications of Theorem~\ref{mainthm}, namely that if $T$ is conjugate to $B + \lambda \operatorname{Id} + R$ or to $F + \lambda \operatorname{Id} + R$ for some finite rank operator $R$ and some $\lambda \in \mathbb{C}$, then $T$ does not possess a closed invariant half-space. Once this is done, we will demonstrate in Section~\ref{SecProof} that if $T$ does not possess a closed invariant half-space then $T$ admits such a conjugacy. 
Our strategy will be to divide the family of operators on $\omega$ into several subfamilies.
First, we will look at the operators for which the set $A$ of indices $n$, such that the nth coordinate of a vector can allow the coordinate of index $0$ to be nonzero under the action of $T$, is infinite.
In this case (Subsection~\ref{SubSecInf}), we will show that either $T$ possesses a closed invariant half-space or is conjugate to $B + \lambda \operatorname{Id} + R$.
We will then look at operators $T$ for which this set $A$ is finite.
In fact, we will even assume that for any $k$, the set of indices $n$ such that $e_n$ sends something non-zero to coordinate $k$ under the action of $T$ is finite. Otherwise, we can reduce ourselves to the previous case.
Therefore we will show in this case (Subsection~\ref{SubSecFin}) that either $T$ has a closed invariant half-space or is conjugate to $F + \lambda \operatorname{Id} + R$.

Finally, in Section~\ref{SecFinal}, we will discuss the consequences of this equivalence. In particular, we will end the proof of Theorem~\ref{mainthm} by deducing the equivalence between having a closed almost-invariant half-space and a closed invariant half-space in the case of $\omega$. We will also give the proof of Theorem~\ref{thmX}.

\section{Invariant half-space and conjugacy}\label{Sec2}

We start this section by introducing the space $\omega$ and several notations.
Let $\mathbb{Z}_{+}$ be the set of non-negative integers. 
We recall that $\omega = \mathbb{C}^{\mathbb{Z}_{+}}$ is the space of all complex sequences.
We then have that $\omega$ is a Fréchet space whose topology is induced by the family of seminorms $(p_j)_{j \geqslant 0}$ where for all $j \geqslant 0$ and for all $x \in \omega$, $p_j(x) = \max_{0 \leqslant n \leqslant j} |x_n|$.
We then also recall that a linear application $T$ defined from $\omega$ to $\omega$ is continuous if for all $j \geqslant 0$, there exist $C_j > 0$ and $J \geqslant 0$ such that for all $x \in \omega$, $p_{j}(Tx) \leqslant C_j p_{J}(x)$.
We denote the set of all linear continuous applications defined on $\omega$ by $L(\omega)$.
Moreover, in this article, we will use the term operator to refer to linear continuous applications. 

Throughout this article, we will need to project certain elements onto specific subspaces.  
We will fix for all $n \geqslant 0$, $P_n$ the projection on the nth coordinate.
We will also fix for all $n \geqslant 0$, $P_{\{0,\ldots,n\}}$ the projection on the subspace $\operatorname{span}\{e_0,\ldots,e_n\}$.
Moreover, we will define the valuation of an element $x \in \omega$ as the first $n \geqslant 0$ such that $P_n x \neq 0$ and we will write $\operatorname{val}(x) = n$.
Then, for example, we have that for all $n \geqslant 0$, $\operatorname{val}(e_n) = n$. 

Since we will need to identify half-spaces in $\omega$, let us begin by noting that if $E$ is a subspace of $\omega$ that possesses an infinite number of elements with different valuations, then $E$ is infinite-dimensional. 
To demonstrate that a subspace $E$ is of infinite codimension, we will use the following result.

\begin{lem}\label{fait}
Let $E$ be a subspace of $\omega$ and $(n_k)_{k \geqslant 0} \subseteq \mathbb{Z}_+$ be an increasing sequence. 
If
$$
n_k - \operatorname{dim}(P_{\{0,\ldots,n_k-1\}}(E)) \underset{k \rightarrow + \infty}{\longrightarrow} + \infty
$$
then $E$ has an infinite codimension.
\end{lem}

\begin{proof}
For all $k \geqslant 0$, let $d_k = \operatorname{dim}(P_{\{0,\ldots,n_k-1\}}(E))$ and $N_k = n_k - d_k$. 
By hypothesis, we know that $(N_k)_{k \geqslant 0}$ converges to infinity.
By definition of $N_k$, we can find linearly independent functionals $y_1^*, \ldots, y_{N_k}^*$ defined from $\operatorname{span}\{e_0,\ldots,e_{n_k-1}\}$ to $\mathbb{K}$ such that 
$$
P_{\{0,\ldots,n_k-1\}}(E) \subseteq \bigcap_{l=1}^{N_k} \operatorname{Ker}y_l^*.
$$
It suffices to consider $(u_1,\ldots,u_{d_k})$ a basis of $P_{\{0,\ldots,n_k-1\}}(E)$, $(v_1,\ldots,v_{N_k})$ such that $(u_1,\ldots,u_{d_k}, v_1,\ldots,v_{N_k})$ forms a basis of $\operatorname{span}\{e_0,\ldots,e_{n_k-1}\}$ and to define for any $1 \leqslant l \leqslant N_k$, 
\[y_l^*\left(\sum_{i=1}^{d_k} x_i u_i + \sum_{i=1}^{N_k} y_i v_i\right) = y_l.\]

Now, if we let for all $1 \leqslant l \leqslant N_k$, $x_l^* = y_l^* \circ P_{\{0,\ldots,n_k-1\}}$, we have that $x_l^*\in \omega^*$ and that these functionals are linearly independent.
We deduce that $\operatorname{codim}(E) \geqslant N_k$ and since $(N_k)_k$ tends to infinity, we conclude that $E$ has an infinite codimension.
\end{proof}

Now that we have a way to identify half-spaces, we want to know under which types of transformations the invariant half-spaces are preserved. We have the following proposition.

\begin{prop}\label{conj preserve DEI}
Let $T,P \in L(\omega)$ such that $P$ is invertible.
Then $T$ possesses a closed invariant half-space if and only if $P^{-1}TP$ possesses a closed invariant half-space.  
\end{prop}

\begin{proof}
Since $PP^{-1}TPP^{-1} = T$, it is enough to show one of the implications.
Let $M$ be a closed invariant half-space of $T$. We first remark that $P^{-1}(M)$ is a closed invariant subspace of $P^{-1}TP$.
Since $P^{-1}$ is an injective operator, we still have that $P^{-1}(M)$ is a closed infinite-dimensional subspace. Moreover, if $(x^*_n)_n$ is a linearly independent sequence in $\omega^*$ such that $x^*_n(M) = \{0\}$ then $(x^*_n \circ P)_n$ is also a linearly independent sequence in $\omega^*$ such that $(x^*_n \circ P)(P^{-1}M) = \{0\}$. We then conclude that $P^{-1}(M)$ is a closed invariant half-space of $P^{-1}TP$.
\end{proof}

In other words, if an operator $T \in L(\omega)$ is conjugate to an operator $S \in L(\omega)$ that possesses a closed invariant half-space, we can then deduce that $T$ also possesses a closed invariant half-space.

The following result will be useful for proving that an operator on $\omega$ is invertible and so to identify some operators conjugate to $T$.
Let us recall that $c_{00}$ is the space of all sequences which have only a finite number of nonzero elements.

\begin{prop}\label{prop invertible}
Let $P$ be a linear application defined from $c_{00}$ to $\omega$. 
If for all $n \in \mathbb{Z}_+$ there exists a unique $k_n$ such that $\operatorname{val}(Pe_{k_n}) = n$ and if $\{k_n:n \geqslant 0\}=\mathbb{Z}_+$, then $P$ can be continuously extended into an invertible operator on $\omega$ by letting
\[P\left(\sum_{n\geqslant 0}x_{k_n} e_{k_n}\right)=\sum_{n\geqslant 0}x_{k_n} Pe_{k_n}.\]
\end{prop}

\begin{proof}
We begin by showing that the extension of $P$ to $\omega$ is well-defined and continuous.
Let $n \geqslant 0$. 
Since $\operatorname{val}(Pe_{k_n}) = n$ and $\{k_n:n\geqslant 0\}=\mathbb{Z}_+$, the extension of $P$ is well-defined on $\omega$. We denote
$$
Pe_{k_n} = (\underbrace{0,\ldots,0}_n,\alpha_{n,n}, \alpha_{n+1,n}, \ldots).
$$
Let $x\in \omega$. Since $\{k_n:n\geqslant 0\}=\mathbb{Z}_+$, we can write $x=\sum_{n \geqslant 0} x_{k_n} e_{k_n}$.
Let $j \geqslant 0$.
We remark that for all $l \geqslant j+1$, $p_j(Pe_{k_l}) = 0$.
We fix $J = \max_{0 \leqslant i \leqslant j} k_i$.
Then, we have 
\begin{align*}
p_j(Px) &= p_j\left(\sum_{n=0}^j x_{k_n} Pe_{k_n}\right) = \sup_{0 \leqslant i \leqslant j} |x_{k_0} \alpha_{i,0} + \ldots + x_{k_i} \alpha_{i,i}| \\
&\leqslant \sup_{0 \leqslant i \leqslant j} (| \alpha_{i,0}| + \ldots +| \alpha_{i,i}|)  p_J(x). 
\end{align*}
We can thus deduce the continuity of $P$.

It remains to show that $P$ is bijective since the conclusion will then follow from the Open mapping theorem.
Let $n \geqslant 0$. We show that there exists $y_n \in \omega$ with $y_n=\sum_{j=n}^{\infty}\beta_j e_{k_j}$ such that $\beta_n\ne 0$ and $Py_n = e_n$.
We know that $\operatorname{val}(Pe_{k_n}) = n$ and $P_n(Pe_{k_n}) = \alpha_{n,n} \neq 0$.
Let $\beta_{n}=\frac{1}{\alpha_{n,n}}$.
Then we have that $P_n(P(\beta_n e_{k_n})) = 1$ and $P_l(P(\beta_n e_{k_n})) = 0$ for all $l < n$.
Either $P(\beta_n e_{k_n})=e_n$ and we let $y_n=\beta_n e_{k_n}$ or we consider the smallest $l_1 > n$ such that $P_{l_1}(P(\beta_n e_{k_n})) \neq 0$.
Since $\operatorname{val}(Pe_{k_{l_1}}) = l_1$, by considering $\beta_{l_1}$ such that $P_{l_1}(\beta_n Pe_{k_n} + \beta_{l_1} Pe_{k_{l_1}}) = 0$ we have that $P_n(\beta_n Pe_{k_n}  + \beta_{l_1} Pe_{k_{l_1}}) = 1$ and $P_l(\beta_n Pe_{k_n}  + \beta_{l_1} Pe_{k_{l_1}}) = 0$ for all $l < n$ and for all $n < l \leqslant l_1$.
Again, either $\beta_n Pe_{k_n}  + \beta_{l_1} Pe_{k_{l_1}}=e_n$ and we consider $y_n = \beta_n e_{k_n} + \beta_{l_1} e_{k_{l_1}}$ or we consider the smallest $l_2 > l_1$ such that $P_{l_2}(\beta_n Pe_{k_n}  + \beta_{l_1} Pe_{k_{l_1}}) \neq 0$ and we continue this way.
Then, we get a vector $y_n = \beta_n Pe_{k_n}  + \sum_{i \geqslant 1} \beta_{l_i} e_{k_{l_i}} \in \omega$ satisfying $Py_n = e_n$.
Then, for $x = \sum_{n \geqslant 0} a_n e_n$, by letting
$$
y = \sum_{n \geqslant 0} a_n y_n,
$$
we have that $y \in \omega$ and $Py = x$.
This shows that $P$ is surjective.

Moreover, $P$ is injective because for any nonzero vector $x = \sum_{j\geqslant n} x_{k_j} e_{k_j} \in \omega$ with $x_{k_n}\ne 0$ then we have  
$$
P_n(Px) = P_n \Big(\sum_{j \geqslant n} x_{k_j} P e_{k_j}\Big) = x_{k_n} \alpha_{n,n}  \ne  0.
$$
\end{proof}

Another important ingredient in the study of almost-invariant half-spaces is given by finite rank operators. These operators on $\omega$ are actually easy to identify.

\begin{prop}\label{rang fini,colonnes}
If $R \in L(\omega)$ is a finite rank operator, then there exists $N \geqslant 0$ such that $Re_n=0$ for every $n \geqslant N$. 
In other words, finite rank operators on $\omega$ are given by matrices admitting only a finite number of nonzero columns.
\end{prop}

\begin{proof}
First, we remark that for every operator $T$ on $\omega$, for each $n \geqslant 0$, there is only a finite number of $m \geqslant 0$ such that $P_n(T e_m) \neq 0$ because by continuity for every $N \geqslant 1$, there exists $C>0$ and $M \geqslant 1$ such that for all $x\in \omega$,
\[p_N(Tx) \leqslant C p_M(x)\]
and thus for every $m>M$, $p_N(Te_m)=0$, i.e. for every $n \leqslant N$, $P_n(Te_m)=0$.

Let $R$ be a finite rank operator. Assume that there exists an increasing sequence $(m_k)$ such that $Re_{m_k}\ne 0$. By the previous remark, we have that $\operatorname{val}(Re_{m_k})$ has to tend to infinity. However, it means that $\operatorname{span}\{R(e_{m_k}) \ ; \ k \geqslant 1\}$ is an infinite-dimensional subspace included in $\operatorname{Im}(R)$. Contradiction.
\end{proof}

We finish this section by remarking that if $T$ is conjugate to $S$ on $\omega$ then for any finite rank operator $R$, there exists a finite rank operator $R'$ such that the operator $T+R$ is conjugate to $S+R'$. This fact will be used several times in this paper.







\section{Case of the Forward shift $F$ and the Backward shift $B$}\label{SecFB}

The goal of this section consists in showing that every operator $T$ on $\omega$ that is conjugate to $F + R + \lambda \operatorname{Id}$ or to $B + R + \lambda \operatorname{Id}$ for some $\lambda \in \mathbb{C}$ and some finite rank operator $R$ does not possess a closed invariant half-space. This will imply that such operators do not possess closed almost-invariant half-space and thus that  $3.$ implies $1.$ and $2.$ in Theorem \ref{mainthm}.

\begin{prop}\label{propF}
Let $T\in L(\omega)$.
Let $\lambda \in \mathbb{C}$ and $R$ a finite rank operator on $\omega$. If $T$ is conjugate to $F + R + \lambda \operatorname{Id}$ then $T$ does not possess a closed invariant half-space.
\end{prop}
\begin{proof}
We first remark that thanks to Proposition \ref{conj preserve DEI}, it is enough to prove that $F + R + \lambda \operatorname{Id}$ does not possess a closed invariant half-space. Moreover, any operator $S$ has a closed invariant half-space if and only if $S + \lambda \operatorname{Id}$ has a closed invariant half-space. Therefore it suffices to show that for every finite rank operator $R$, the operator $F+R$ does not possess a closed invariant half-space.

Let $R$ be a finite rank operator and $M$ a closed invariant subspace of infinite dimension for $F+R$. We have to show that $M$ cannot be a half-space.

By Proposition~\ref{rang fini,colonnes}, we know that there exists $N \geqslant 0$ such that $Re_n=0$ for every $n \geqslant N$. We let
$$
M_N = \{ x \in \omega \ : \operatorname{val}(x) \geqslant N \}.
$$
We then remark that $M_N$ is a $F$-invariant closed subspace of finite codimension included in the kernel of $R$.

Since $M_N$ is a closed subspace of finite codimension, we deduce that $M_N \cap M$ is a closed infinite-dimensional subspace. Moreover, since $M$ is invariant for $F+R$, $M_N$ is invariant for $F$ and $M_N\subseteq \ker R$, we deduce that $M\cap M_N$ is a closed invariant subspace for $F$.

Let $x \in M_N \cap M$ be a nonzero vector and let $n=\operatorname{val}(x)\geqslant N$.  
Then for all $k \geqslant 0$, $F^kx\in M_N\cap M\subseteq M$ and $\operatorname{val}(F^kx)=n+k$.
Since $M$ contains vectors of valuation $m$ for every $m \geqslant n$, we deduce that $M_n$ is included in $M$ and thus $M$ is not a half-space.


\end{proof}

We now deal with the case of the Backward shift.
\begin{prop}\label{PropB}
Let $T\in L(\omega)$.
Let $\lambda \in \mathbb{C}$ and $R$ a finite rank operator on $\omega$. If $T$ is conjugate to $B + R + \lambda \operatorname{Id}$ then $T$ does not possess a closed invariant half-space.
\end{prop}
\begin{proof}
As in the previous proposition, it suffices to show that for every finite rank operator $R$, the operator $B+R$ does not possess a closed invariant half-space.

Let $R$ be a finite rank operator.
Again there exists $N \geqslant 0$ such that $$M_{N}:= \{ x \in \omega \ : \operatorname{val}(x)\geqslant N \} \subset \ker R.$$

Let $M$ be a closed invariant subspace for $B+R$ of infinite dimension. We show that $M_N$ is included in $M$ and thus that $M$ is not a half-space. To this end, it suffices to show that for every $k\geqslant N$, there exists $x\in M$ such that $\operatorname{val}(x)=k$. 
Let $k \geqslant N$. Since $M_k$ is a closed subspace of finite codimension, we have that $M_k \cap M$ is a closed subspace of infinite dimension and we can thus find  a nonzero vector $x \in M\cap M_k$. Let $n=\operatorname{val}(x)\geqslant k$.
Since $M$ is invariant for $B+R$ and $M_N\subseteq \ker R$, we have that 
\[B^{n-k}x = (B+R)^{n-k}x\in M\]
and since $\operatorname{val}(B^{n-k}x)=k$, we get the desired result.

\end{proof}

This allows us to conclude that if there exists a finite rank operator $R$ and $\lambda \in \mathbb{C}$ such that $T$ is conjugate to $B+R+\lambda \operatorname{Id}$ or to $F+R+\lambda \operatorname{Id}$, then $T$ does not possess a closed invariant half-space. In fact, we can even deduce that such an operator $T$ does not possess a closed almost-invariant half-space

\begin{thm}
Let $T \in L(\omega)$. 
If there exist $\lambda \in \mathbb{C}$ and an operator $R \in L(\omega)$ of finite rank  such that $T$ is conjugate to $B+ \lambda \operatorname{Id}  + R$ or to $F+ \lambda \operatorname{Id}+ R$, then $T$ does not possess a closed almost-invariant half-space.
\end{thm}
\begin{proof}
It suffices to remark that if $T$ is conjugate to $B+ \lambda \operatorname{Id}  + R$ for some $\lambda \in \mathbb{C}$ and some  operator of finite rank $R \in L(\omega)$ then for every operator of finite rank $R'$, $T+R'$ is conjugate to $B+ \lambda \operatorname{Id}  + R''$ for some  operator of finite rank $R''$ and thus by Proposition~\ref{PropB}, $T+R'$ does not possess a closed invariant half-space. We conclude in the same way if $T$ is conjugate to $F+ \lambda \operatorname{Id}  + R$ by using Proposition~\ref{propF}.
\end{proof}

We have thus shown that $3.$ implies $2.$ and $1.$ in Theorem \ref{mainthm}.

\section{Proof of Theorem~\ref{mainthm}}\label{SecProof}

We now present the most significant implication of Theorem \ref{mainthm}.
We will prove in this section that $2.$ implies $3.$. 
That is to say, we will show that if an operator $T$ defined on $\omega$ does not posses a closed invariant half-space then there exist a finite rank operator $R$ and $\lambda \in \mathbb{C}$ such that $T$ is conjugate to $B + R + \lambda \operatorname{Id}$ or to $F + R + \lambda \operatorname{Id}$.

The strategy to prove this implication is to divide the family of operators into different subfamilies depending on the set
\[A = \{n \geqslant 1 \  ; \  P_0 T^{l}e_n \neq 0 \ \text{for some $l$}\}.\]

We will begin by focusing on the operators for which the set $A$ is infinite. We will then show that either $T$ possesses a closed invariant half-space or is conjugate to $B + \lambda \operatorname{Id} + R$ with $\lambda \in \mathbb{C}$ and $R$ a finite rank operator.

In a second part, we will consider all the operators $T$ for which the set $A$ is finite.
In this case, we will even assume that for any $k$, the set $\{n \ne k \ ; \ P_k T^{l}e_n \neq 0 \ \text{for some $l$}\}$ is finite. 
Otherwise, we can revert to the first case.
We will then show in this case that either $T$ has a closed invariant half-space or is conjugate to $F + \lambda \operatorname{Id} + R$ with $\lambda \in \mathbb{C}$ and $R$ a finite rank operator.

\subsection{A is infinite}\label{SubSecInf}
For this section, we now fix an operator $T$ for which the set $A$ is infinite and we let for each $l \geqslant 1$,
$$
A_l = \{n \geqslant 1 \ ; \ P_0 T^l e_n \neq 0 \text{ and } P_0 T^{l'} e_n = 0 \text{ for all } l' < l \}.
$$
We then have that $A = \bigcup_{l \geqslant 1} A_l$.
We can also remark that for all $l \geqslant 1$, $A_l$ is a finite set. 
Indeed, by continuity of $T$, for all $l \geqslant 1$, there exist $C>0$ and $J\geqslant 1$ such that for $n \geqslant 1$,
$$
p_0(T^l e_n) \leqslant C p_J(e_n)
$$
and thus $A_l \subseteq [1,J]$.

We now consider the two following possibilities:  either $A$ is an infinite set that is not cofinite or $A$ is a cofinite set.

We start by showing that if $A$ is an infinite set that is not cofinite, then $T$ possesses a closed invariant half-space.

\begin{prop}\label{notcofinite}
Let $T\in L(\omega)$. If $A$ is an infinite set that is not cofinite then $T$ possesses a closed invariant half-space.
\end{prop}
\begin{proof}
Since $A^c$ is infinite, $M = \overline{\text{span}}\{T^l e_n : n \in A^c\backslash\{0\},\ l \geqslant 0\}$ is a $T$-invariant closed subspace of infinite dimension.
It remains to show that $M$ is of infinite codimension.
Since $A$ is infinite and since each set $A_l$ is finite, we know that there exists a strictly increasing sequence $(l_m)_m$ such that $A_{l_m} \neq \emptyset$ for all $m \geqslant 1$. Let $(x^*_{m})_m$ be the sequence defined by $x^*_{m}(x) = P_0 T^{l_m}(x)$ for all $x \in \omega$.
We remark that this sequence is linearly independent because for every  $m \geqslant 1$, by considering $n \in A_{l_m}$, we have $x^*_{m}(e_n) \neq 0$ and for all $k < m$, $x^*_{k}(e_n) = 0$. 
Moreover, we have $M\subset \ker x^*_{m}$ for all $m \geqslant 1$ since for all $n \in A^c\backslash\{0\}$ and for all $l \geqslant 0$, $P_0 T^{l} e_n = 0$. 
We can thus deduce that $M$ is a half-space. 
\end{proof}

We are now investigating the sets $A_l$ in order to deal with the case where $A$ is cofinite. 
We start by showing that by conjugacy, if $A$ is infinite, we can assume that each set $A_l$ is non-empty.

\begin{prop}\label{|A_l|>= 1}
If $A$ is infinite then $T$ is conjugate to an operator for which for all $l\geqslant 1$, we have $|A_l| \geqslant 1$.
\end{prop}

\begin{proof}
We recall that since $A$ is infinite, an infinite number of sets $A_l$ is non-empty. We can thus consider a strictly increasing sequence $(l_k)_{k \geqslant 1}$ and a sequence $(n_k)_{k \geqslant 1}$ such that $n_k\in A_{l_k}$. Let $l \geqslant 1$. 
We let $v_l = T^{l_1-l} e_{n_1}$ if $l \leqslant l_1$ and $v_l=T^{l_k-l}e_{n_k}$ if $l_{k- 1} < l \leqslant l_k$.

We then get a sequence $(v_l)_{l \geqslant 1}$ such that for all $l \geqslant 1$, we have
$$
P_0 T^l v_l \neq 0 \text{  and  } P_0 T^{l'} v_l = 0 \text{ for all } l' < l.
$$

We now show that we can modify this sequence $(v_l)_{l \geqslant 1}$ such that the value of $P_0 T^{l'} v_l$ does not change for any $l' \leqslant l$ but the elements $v_l$ have different valuations. 


Let $I_0=\mathbb{Z}_+$. There are three possibilities:
\begin{itemize}
\item If $\{l\in I_0: v_{l,1}\ne 0\}$ is empty, then we do not change the vectors $v_l$ and we let $I_1=I_0$.  
\item If $\{l\in I_0: v_{l,1}\ne 0\}$ is finite, then we consider the largest $k\in I_0$ such that $v_{k,1} \neq 0$ and for all $l < k$ with $l\in I_0$, we replace $v_l$ by $v_l - \frac{v_{l,1}}{v_{k,1}} v_k$ so that $P_1 (v_l -\frac{v_{l,1}}{v_{k,1}}v_k) = 0$.
After these modifications, $v_k$ is the only element on valuation $1$ and we let $I_1=I_0\backslash\{k\}$.
\item If $\{l\in I_0: v_{l,1}\ne 0\}$ is infinite,  then we consider $(m_k)_{k\geqslant 1}$ the increasing enumeration of this set. 
Let $k \geqslant 1$ and $m_{k-1} \leqslant l < m_k$. 
We then replace $v_l$ by $v_l - \frac{v_{l,1}}{v_{m_k,1}} v_{m_k}$ so that
$P_1 (v_l - \frac{v_{l,1}}{v_{m_k,1}} v_{m_k}) = 0$ and we let $I_1=I_0$.
After these modifications, no $v_l$ with $l\in I_0$ has a valuation equals to $1$.
\end{itemize}

After this step, we have thus a new family $(v_l)_{l\in I_0}$ such that the values of $P_0 T^{l'} v_l$ have not changed for any $l' \leqslant l$ and such that for all $l\in I_1$, the vector $v_l$ has a valuation bigger than 2. 
Notice that we also have $|I_1\backslash I_0|\leqslant 1$. 
We then repeat these modifications for $I_1$ and the coordinate of index $2$. 
By continuing in this way for each coordinate, we get at the end a new sequence $(v_l)_{l \geqslant 1}$ such that for all $l \geqslant 1$, we still have
$$
P_0 T^l v_l \neq 0 \text{  and  } P_0 T^{l'} v_l = 0 \text{ for all } l' < l
$$
and such that for each $k \in \mathbb{Z}_+$, there is now at most one element $v_l$ of valuation $k$. It should be noted that the successive modifications of $v_l$ will converge since at the kth step we will only potentially modify the coordinates of indices bigger than $k$ and that this limit will never be zero since the value of $P_0T^lv_l$ never changes.  



Finally, we consider a sequence $(w_l)_{l \geqslant 0}$ with $w_0=e_0$, containing each element $v_l$ for $l \geqslant 1$ and containing $e_k$ for each $k \geqslant 1$ such that there exists no $l$ such that $\operatorname{val}(v_l) = k$. 
We now consider $P$ the linear application such that $P(e_n) = w_n$. 
By Proposition \ref{prop invertible}, we have that $P$ gives us an invertible operator on $\omega$ and thus $T$ is conjugate to $S=P^{-1}TP$. 
Moreover, the sets $A_l$ associated to $S$ satisfy for each $l \geqslant 1$
$$
A_l = \{n \geqslant 1 \ ; \ P_0 S^l e_n \neq 0 \text{ and } P_0 S^{l'} e_n = 0 \text{ for all } l' < l \}\ni k
$$ 
where $k$ is such that $P(e_k) = v_l$.
\end{proof}

Since conjugacy does not change the existence of closed invariant half-spaces, we can assume from now that for all $l \geqslant 1$, $|A_l| \geqslant 1$. We now aim to show that depending on if there exists an infinite number of $l \geqslant 1$ such that $|A_l| \geqslant 2$ or not, the operator $T$ will possess a closed invariant half-space or will be conjugate to $B + R + \lambda \operatorname{Id}$ for some finite rank operator $R$ and some $\lambda \in \mathbb{C}$.

\begin{prop} \label{|A_l| >= 2}
If for all $l\geqslant 1$, we have $|A_l| \geqslant 1$ and if there exists an infinite number of $l \geqslant 1$ such that $|A_l| \geqslant 2$, then $T$ possesses a closed invariant half-space.
\end{prop}

\begin{proof}
Let $(l_k)_k$ be an increasing enumeration of indices $l$ such that $|A_{l}| \geqslant 2$. 
We have then that for $l \notin \{l_k : k \geqslant 1\}$, $|A_l| = 1$ and we denote $A_l = \{n_l\}$.
Let $(e_{n_{l_k}})_{k \geqslant 1}$ and $(e_{m_{l_k}})_{k \geqslant 1}$ such that $\{n_{l_k}, m_{l_k}\} \subseteq A_{l_k}$ with $n_{l_k}\ne m_{l_k}$.
Then,
$$
P_0 T^{l_k} e_{n_{l_k}} \neq 0 \text{ and } P_0 T^{l} e_{n_{l_k}} = 0 \text{ for all } l < l_k
$$
and 
$$
P_0 T^{l_k} e_{m_{l_k}} \neq 0 \text{ and } P_0 T^{l} e_{m_{l_k}} = 0 \text{ for all } l < l_k.
$$
Let $M = \{x \in \omega : P_0T^{l}(x) = 0 \text{ for all } l \geqslant 0\}$. 
We have that $M$ is a $T$-invariant closed subspace of infinite codimension.
Indeed, we consider the sequence $(x^*_{l})_{l \geqslant 1}$ defined by $x^*_{l}(x) = P_0 T^{l}(x)$ for all $x \in \omega$.
This sequence is linearly independent.
Otherwise, we would have that for a certain $l \geqslant 2$, $x^*_{l} = \sum_{k=1}^{l-1} \lambda_k x^*_{k}$ where $\lambda_1, \ldots, \lambda_{l-1} \in \mathbb{C}$. 
Which is impossible because for $n_l \in A_{l}$, we have that $x^*_{l}(e_{n_l}) = P_0(T^l e_{n_l}) \neq 0$ and for all $k < l$, $x^*_{k}(e_{n_l}) = P_0(T^k e_{n_l}) = 0$. 
Moreover, we have that $x^*_{l}(M) = \{0\}$ for all $l \geqslant 1$. 

We now show that $M$ is of infinite dimension.
Let $k \geqslant 1$. 
We fix $\alpha$ such that $P_0 T^{l_k} e_{n_{l_k}} - \alpha P_0 T^{l_k} e_{m_{l_k}} = 0$. 
We also fix $(\beta_l)_{l \geqslant l_k + 1}$ such that
$$
\beta_{l_k + 1} = \frac{-P_0 T^{l_k + 1} e_{n_{l_k}} + \alpha P_0 T^{l_k + 1} e_{m_{l_k}}}{P_0 T^{l_k + 1} e_{n_{l_k + 1}}}
$$
and for all $l \geqslant l_k + 2$,
$$
\beta_{l} = \frac{-P_0 T^{l} e_{n_{l_k}} + \alpha P_0 T^{l} e_{m_{l_k}} - \sum_{j = l_k + 1}^{l-1} \beta_j P_0 T^l e_{n_j}}{P_0 T^{l} e_{n_l}}.
$$
We consider $v_k = e_{n_{l_k}} - \alpha e_{m_{l_k}} + \sum_{j = l_k + 1}^{\infty} \beta_j e_{n_j}\in \omega$. 
We then have that for every $l \geqslant 0$, $P_0(T^l v_k) = 0$. 
Indeed, for all $l \geqslant l_{k}+1$, we have by definition of $\beta_l$
\begin{align*}
P_0(T^l v_k) &=
P_0T^le_{n_{l_k}} - \alpha P_0 T^le_{m_{l_k}} + \sum_{j = l_k + 1}^{\infty} \beta_j P_0T^l e_{n_j}\\
&=
P_0T^le_{n_{l_k}} - \alpha P_0 T^le_{m_{l_k}} + \beta_lP_0T^l e_{n_l}+\sum_{j = l_k + 1}^{l-1} \beta_j P_0T^l e_{n_j} = 0.
\end{align*}
Moreover, for all $l < l_{k}$, we have $P_0(T^l v_k) = 0$ since for all $j \geqslant l_k$, we have $P_0 T^l e_{n_j} = P_0 T^l e_{m_{l_k}} = 0$.
Finally, for $l=l_k$, we have by definition of $\alpha$,
$$
P_0(T^{l_k} v_k) =
P_0T^{l_k} e_{n_{l_k}} - \alpha P_0 T^{l_k} e_{m_{l_k}} = 0.
$$

We then have $(v_k)_{k \geqslant 1} \subseteq M$. 
Moreover, we have $P_{m_{l_k}} v_k \neq 0$ and for all $l \neq k$, $P_{m_{l_k}} v_l = 0$. 
We can then deduce that the vectors $v_k$ are linearly independent.
As a consequence, $M$ is infinite-dimensional and thus a closed invariant half-space.
\end{proof}

We now show that in the remaining case, the operators are conjugate to $B + R + \lambda \operatorname{Id}$ for $R$ a finite rank operator and for $\lambda \in \mathbb{C}$.
We have the following result.

\begin{prop} \label{|A_l| = 1}
If $A$ is cofinite and if there exists $l_0$ such that $|A_l| = 1$ for all $l \geqslant l_0$, then there exist a finite rank operator $R$ and $\lambda \in \mathbb{C}$ such that $T$ is conjugate to $B + R + \lambda \operatorname{Id}$.
\end{prop}

\begin{proof}
Let $G = \bigcup_{l=0}^{l_0-1} A_l \cup (\mathbb{Z}_+ \backslash A)$ where $A_0 = \{0\}$ and $P_G$ the projection on $\operatorname{span}\{e_n ; n \in G\}$. 
Since $G$ is a finite set, there exists by continuity of $T$ an integer $L_0 \geqslant l_0$ such that for every $n\in \bigcup_{l \geqslant L_0}A_{l}$, we have $P_G(Te_n)=0$.

We now consider $C=\bigcup_{l=0}^{L_0-1} A_l \cup (\mathbb{Z}_+ \backslash A) $
and the infinite set $D = \bigcup_{l \geqslant L_0} A_l$.
We also let $\operatorname{span}(C) = \operatorname{span}\{e_n ; n \in C\}$ and $\operatorname{\overline{span}}(D) = \operatorname{\overline{span}}\{e_n ; n \in D\}$ so that
\[\omega=\operatorname{span}(C)\oplus \operatorname{\overline{span}}(D).\] 
In other words, for all $x \in \omega$, there exists a unique $y \in \operatorname{span}(C)$ and a unique $z \in \operatorname{\overline{span}}(D)$ such that $x = y + z$. We will write $P_C x=y$ and $P_D x=z$ We can then write
$$
Tx = T(P_Cx) + T(P_Dx) = T_C (P_Cx) + T_D(P_Dx)
$$ 
where $T_C : \operatorname{span}(C) \rightarrow \omega : y \rightarrow Ty$ and $T_D : \operatorname{\overline{span}}(D) \rightarrow \omega : z \rightarrow Tz$.
Since $\operatorname{span}(C)$ is of finite dimension,  $T_C$ is a finite rank operator. 
Moreover, for $z \in \operatorname{\overline{span}}(D)$, we can write
$$
T_D z = T_{DC} z + T_{DD} z
$$
where $T_{DC} : \operatorname{\overline{span}}(D) \rightarrow \operatorname{span}(C) : z \rightarrow P_C Tz$ is a finite rank operator and $T_{DD} : \operatorname{\overline{span}}(D) \rightarrow \operatorname{\overline{span}}(D) : z \rightarrow P_D Tz$.
We then have that 
$$
Tx = Rx + Sx
$$ 
where $R$ is the finite rank operator given by $Rx = T_C (P_C x) + T_{DC} (P_D x)$ and $S$ is the linear operator defined by $Sx = 0$ if $x \in \operatorname{span}(C)$ and defined by $Sx = T_{DD}x$ if $x \in \operatorname{\overline{span}}(D)$. 

Let us focus on the operator $T_{DD}$. 
For every $l \geqslant l_0$, we know that $A_{l}=\{n_l\}$ for some $n_l$. For what follows, let us denote $v_l = e_{n_{L_0+l}}$ for all $l \geqslant 0$. We first show that the action of $T_{DD}$ can be represented by the following matrix
$$
\bordermatrix{
     & v_0 & v_1 & v_2 & v_3 \cr
    v_0 & \alpha_{0,0} & \alpha_{0,1} & 0 & 0 &\cdots \cr
    v_1 & \alpha_{1,0} & \alpha_{1,1} & \alpha_{1,2} & 0 &\cdots \cr
     & \vdots & \vdots & \vdots & \ddots \cr
     & \vdots & \vdots & \vdots & \vdots }
$$
where $\alpha_{l,l+1} \neq 0$ for all $l \geqslant 0$.
Indeed, for every $l \geqslant 0$, we know that $P_G(T v_l)=0$ by definition of $L_0$ and thus that $T v_l=\sum_{j=l_0}^{\infty}a_j e_{n_j}$. 
However, since $P_0(T^j v_l)=0$ for every $j<L_0+l$, we deduce that $a_j=0$ for every $l_0 \leqslant j < L_0+l-1$ (as otherwise by considering the smallest $l_0 \leqslant j<L_0+l-1$ such that $a_j\ne 0$, we would get $P_0(T^{j+1} v_l)=P_0(T^{j} a_je_{n_j})\ne 0$). Moreover, since $P_0(T^{L_0+l} v_l)\ne 0$, we deduce that $a_{L_0+l-1}$ is nonzero as desired.\\


For convenience, we will now see $\operatorname{\overline{span}}(D)$ as a copy of $\omega$ related to the basis $(v_l)_{l \geqslant 0}$ and we will denote by $P_l$ the projection given by $P_l(\sum_{j \geqslant 0}x_jv_j)=x_l$. 
Our goal consists in finding a sequence $(u_l)_{l \geqslant 0}$ with $u_0=v_0$ such that $T_{DD} u_l = u_{l-1}$ for all $l \geqslant 1$ and such that $\operatorname{val}(u_l) = l$ where the valuation is here considered relatively to the sequence $(v_l)_{l \geqslant 0}$ . \\

Let $u_0 = v_0$.
We start by looking for a vector $u_1$ such that $T_{DD}u_1 = u_0$ and $\operatorname{val}(u_1) = \operatorname{val}(v_1)$.
Since $\alpha_{0,1} \neq 0$ and since $P_1 T_{DD} v_2 = \alpha_{1,2} \neq 0$, we can consider $y_1 = \frac{1}{\alpha_{0,1}} v_1 - \frac{\alpha_{1,1}}{\alpha_{0,1}\alpha_{1,2}} v_2$ so that $P_0 T_{DD} y_1 = 1$ and $P_1 T_{DD} y_1 = 0$. 

Let $k \geqslant 2$. 
We assume that we have $y_{k-1} = \frac{1}{\alpha_{0,1}} v_1 - \sum_{l=2}^{k} c_{1,l} v_l$ for some $c_{1,2}, \ldots,c_{1,k} \in \mathbb{C}$ such that $P_0 T_{DD} y_{k-1} = 1$ and $P_1 T_{DD} y_{k-1} = \ldots = P_{k-1} T_{DD} y_{k-1} = 0$.
We can now fix 
$$
y_k = y_{k-1} - \frac{P_k T_{DD} y_{k-1}}{\alpha_{k,k+1}} v_{k+1}.
$$
Then we have that $y_k = \frac{1}{\alpha_{0,1}} v_1 - \sum_{l=2}^{k+1} c_{1,l} v_l$ with $c_{1,k+1} = \frac{P_k T_{DD} y_{k-1}}{\alpha_{k,k+1}}$.
We also have
$$
P_0 T_{DD} y_k = P_0 T y_{k-1} - \frac{P_k T_{DD} y_{k-1}}{P_k T_{DD} v_{k+1}} P_0 T_{DD} v_{k+1} = 1
$$
since $P_0 T_{DD} v_{k+1} = 0$ and for any $1 \leqslant l \leqslant k$, $P_l T_{DD} y_k=0$.
By letting $u_1 = \lim_{k \to \infty} y_k = \frac{1}{\alpha_{0,1}} v_1 - \sum_{l \geqslant 2} c_{1,l} v_l$, we then get that $T_{DD}u_1 = u_0$ and that $\operatorname{val}(u_1) = 1$.

Let $j \geqslant 2$. 
Assume that we have already found $u_1, \ldots, u_{j-1}$ such that for every $1 \leqslant l \leqslant j-1$, $Tu_l = u_{l-1}$  and $u_l = c_{l,l} v_l + \sum_{j \geqslant l+1} c_{l,j} v_j$ for some $c_{j,l} \in \mathbb{K}$.
We then have
\begin{align*}
T_{DD}v_j &= \sum_{l \geqslant j-1} \alpha_{l,j} v_l = \alpha_{j-1,j} v_{j-1} + \sum_{l \geqslant j} \alpha_{l,j} v_l \\
&= \frac{\alpha_{j-1,j}}{c_{j-1,j-1}} \big(c_{j-1,j-1}v_{j-1} + \sum_{l \geqslant j} c_{j,l} v_l \big) + \sum_{l \geqslant j} (\alpha_{l,j}- \frac{\alpha_{j-1,j}}{c_{j-1,j-1}} c_{j,l}) v_l \\
&=\frac{\alpha_{j-1,j}}{c_{j-1,j-1}} u_{j-1} + \sum_{l \geqslant j} \beta_{l,j} v_l
\end{align*}
where $\beta_{l,j} = \alpha_{l,j}- \frac{\alpha_{j-1,j}}{c_{j-1,j-1}} c_{j,l} $. We can then repeat the previous construction to get $u_j$.

Let $P$ be the linear application on $\omega$ defined for all $l \geqslant 0$ by $P(v_l) = u_l$ and for all $n\in C$ by $P(e_n) = e_n$.
The operator $P$ is then invertible and we deduce that $S$ is conjugate to the operator whose matrix is given by
\[\left(\vcenter{\hbox{
    \begin{tikzpicture}[x=0.7cm, y=-0.4cm]
        \draw[opacity=0] (0.5, 0.5) rectangle (6.5, 6.5);
        \node[font=\LARGE] at (1.5, 1.5) {$0$};
        \node[font=\LARGE] at (4.5, 1.5) {$0$};
        \node[font=\LARGE] at (1.5, 4.5) {$0$};
        \draw (2.5, 6.5) -- (2.5,0);
        \draw (0, 2.5) -- (6.5,2.5);
        \node at (3,3) {$\alpha_{0, 0}$};
        \node at (4,3) {$1$};
        \node at (5,4) {$1$};
        \draw[dots] (5.5, 4.5) -- (6.5, 5.5);
        \node at (3,4) {$\alpha_{1,0}$};
        \node at (3,5) {$\alpha_{2,0}$};
        \draw[dots] (3, 5.5) -- (3, 6.5);
        \node at (4,4) {$0$};
        \node at (4, 5) {$0$};
        \draw[dots] (4, 5.5) -- (4, 6.5);
        \node at (5,5) {$0$};
        \draw[dots] (5.5, 5.5) -- (6.5, 6.5);
        \draw[dots] (5, 5.5) -- (5, 6.5);
        \node at (5, 3) {$0$};
        \draw[dots] (5.5, 3) -- (6.5, 3);
        \draw[dots] (5.5, 3.5) -- (6.5, 4.5);
    \end{tikzpicture}
}}\right).\]
Moreover, since $T=R+S$ where $R$ is a finite rank operator and since the matrix of any finite rank operator has only a finite number of nonzero columns (Proposition~\ref{rang fini,colonnes}), we deduce that $T$ is conjugate to an operator whose matrix is given by 
\[\left(\vcenter{\hbox{
    \begin{tikzpicture}[x=0.7cm, y=-0.4cm]
        \draw[opacity=0] (-0.5,-0.5) rectangle (4.5,4.5);
        \draw (1.5,-0.5) -- (1.5, 5.5);
        \draw (-0.5,1.5) -- (5.5, 1.5);
        \node[font=\LARGE] at (0.5,3.5) {$*$};
        \node[font=\LARGE] at (0.5,0.5) {$*$};
        \node at (2, 2) {$0$};
        \node at (3, 3) {$0$};
        \node at (2, 3) {$0$};
        \draw[dots] (2, 3.5) -- (2, 4.5);
        \draw[dots] (3, 3.5) -- (3, 4.5);
        \draw[dots] (3.5, 3.5) -- (4.5, 4.5);
        \node at (2, 1) {$1$};
        \node at (3, 1) {$0$};
        \draw[dots] (3.5, 1) -- (4.5, 1);
        \node at (3, 2) {$1$};
        \draw[dots] (3.5, 2.5) -- (4.5, 3.5);
        \node[font=\LARGE] at (3.5,0) {$0$};
    \end{tikzpicture}
}}\right).\]
Such a matrix is equal to $B + R'$ for some finite rank operator $R'$ and we can thus conclude.
\end{proof}

\begin{rem}\label{no Id for B}
At the end of the previous proof we get that $T$ is conjugate to $B+R$ where $R$ is a finite rank operator and the term $\lambda \operatorname{Id} $ is not necessary. It comes from the fact that each operator $B+\lambda \operatorname{Id}$ is conjugate to $B+R$ for some finite rank operator $R$.
\end{rem}

In conclusion, we obtain the following statement.

\begin{thm}
Let $T \in L(\omega)$. If there exists $k \geqslant 0$ such that the set 
\[\{n \ne k \  ; \  P_k T^{l}e_n \neq 0 \ \text{for some $l$}\} \ \text{is infinite,}\]
then either $T$ possesses a closed invariant half-space or there exist $\lambda \in \mathbb{C}$ and an operator $R \in L(\omega)$ of finite rank such that $T$ is conjugate to $B+ \lambda \operatorname{Id}  + R$.
\end{thm}
\begin{proof}
Let $k \geqslant 0$ such that the set 
\[\{n \ne k \  ; \  P_k T^{l}e_n \neq 0 \ \text{for some $l$}\}\]
is infinite. 
By considering the invertible operator $P$ defined by $Pe_0=e_k$, $Pe_k=e_0$ and $Pe_n=e_n$ for all $n\notin \{0,k\}$, we deduce that $T$ is conjugate to an operator $S$ such that \[A:=\{n \geqslant 1 \  ; \  P_0 S^{l}e_n \neq 0 \ \text{for some $l$}\} \ \text{is infinite.}\]
It then follows from Proposition~\ref{notcofinite} that if $A$ is not cofinite then $S$ possesses a closed invariant half-space and thus by Proposition~\ref{conj preserve DEI}, $T$ possesses a closed invariant half-space. 
Moreover, if we're going to conjugate again, we can assume thanks to Proposition~\ref{|A_l|>= 1} that for every $l \geqslant 1$, 
$$
|A_l| = |\{n \geqslant 1 \ ; \ P_0 S^l e_n \neq 0 \text{ and } P_0 S^{l'} e_n = 0 \text{ for all } l' < l \}| \geqslant 1 
$$ and it follows from Proposition~\ref{|A_l| >= 2} and Proposition~\ref{|A_l| = 1} that either $S$ possesses a closed invariant half-space or is conjugate to $B+ \lambda \operatorname{Id} + R$ for some $\lambda \in \mathbb{C}$ and some operator $R$ of finite rank. 
We then get the same conclusion for $T$.
\end{proof}

\subsection{$A$ is finite}\label{SubSecFin}

If $A$ is finite, it means that only a finite number of elements $e_n$ satisfies $P_0T^le_n\ne 0$ for some $l \geqslant 1$.
We can in fact assume that it is the case for any projection $P_k$ because otherwise, if we were to swap $e_k$ and $e_0$, we would be in the previous case. Therefore, it remains to deal with the operators $T$ such that for all $k \geqslant 0$, the set $C_k$ is finite where 
\[C_k= \{n \ne k \  ; \  P_k T^{l}e_n \neq 0 \ \text{for some $l$}\}.\]
Our first goal is actually to show that such an operator $T$ is conjugate to an operator whose the associated matrix is a lower triangular matrix. This will be done by induction on the sets $C_k$.

We first let $u_0 = e_0$ and we consider $C_0$. If $C_0$ is empty, it means that the matrix associated to $T$ is given by 
\[\begin{pmatrix}
    * & 0 & 0 & \cdots \cr
    * & * & * & \cdots \cr
    * & * & * & \cdots \cr
    \vdots & \vdots & \vdots & \ddots  
\end{pmatrix}.\]
and we proceed to the next step. If $C_0$ is not empty and since $C_0$ is finite, there exists $k_0 \geqslant 1$ such that $|C_0| = k_0$.
Let $l \geqslant 1$. We denote $C_{0,l} = \{n \geqslant 1 \ ; \ P_0 T^l e_n \neq 0 \text{ and } P_0 T^{l'} e_n = 0 \text{ for all } l' < l \}$.
We remark that $C_{0,1}$ is then not empty and we note $C_{0,1} = \{i_{1,1}, \ldots,i_{1,k_1}\}$ with $i_{1,1} < \ldots < i_{1,k_1}$.
We fix $u_{i_{1,k_1}} = e_{i_{1,k_1}}$ and for all $1 \leqslant l < k_1$, we fix $u_{i_{1,l}} = e_{i_{1,l}} - \frac{P_0 Te_{i_{1,l}}}{P_0 Te_{i_{1,k_1}}} e_{i_{1,k_1}}$ such that $P_0 Tu_{i_{1,l}} = 0$ and such that $\operatorname{val}(u_{i_{1,l}}) = i_{1,l}$.
We now consider this new basis where the vectors $(e_{i_{1,l}})_{1 \leqslant l \leqslant k_1}$ have been replaced by $(u_{i_{1,l}})_{1 \leqslant l \leqslant k_1}$.
We still have that $C_0$ is finite with respect to this new basis and we have that the set $C_0$ (with respect to this new basis) has at most $k_0$ elements. 
However, we now have that $C_{0,1} = \{i_{1,k_1}\}$.

If $C_{0,2}$ is empty then it means that for every $n\notin\{0,i_{1,k_1}\}$ we have $P_0(T^le_n)=0$ for every $l\geqslant 0$ and thus, by swapping $e_{i_{1,k_1}}$ with $e_1$, we will get that the matrix associated to $T$ is given by 
\[\begin{pmatrix}
    * & * & 0 & \cdots \cr
    * & * & 0 & \cdots \cr
    * & * & * & \cdots \cr
    \vdots & \vdots & \vdots & \ddots  
\end{pmatrix}\]
and we proceed to the next step. 

If $C_{0,2}$ is not empty, we note $C_{0,2} = \{i_{2,1}, \ldots,i_{2,k_2}\}$ with $i_{2,1} < \ldots < i_{2,k_2}$.
We have that $P_{i_{1,k_1}} Te_{i_{2,l}} \neq 0$ for all $1 \leqslant l \leqslant k_2$.
We fix $u_{i_{2,k_2}} = e_{i_{2,k_2}}$ and for all $l < k_2$, we fix $u_{i_{2,l}} = e_{i_{2,l}} - \frac{P_{i_{1,k_1}} Te_{i_{2,l}}}{P_{i_{1,k_1}} Te_{i_{2,k_2}}} e_{i_{2,k_2}}$ such that $P_{i_{1,k_1}} Tu_{i_{2,l}} = 0$ and such that $\operatorname{val}(u_{i_{2,l}}) = i_{2,l}$.
We consider now the new basis where we replace $(e_{i_{2,l}})_{1 \leqslant l \leqslant k_2}$ by $(u_{i_{2,l}})_{1 \leqslant l \leqslant k_2}$.
We still have that $C_0$ is finite with respect to this new basis and we have that the set $C_0$ (with respect to this new basis) has at most $k_0$ elements. 
However, we now have that $C_{0,1} = \{i_{1,k_1}\}$ and $C_{0,2} = \{i_{2,k_2}\}$.
Since $C_0$ has always at most $k_0$ element, this proceed will come to an end with some $l \leqslant k_0$ such that $C_{0,l} = \emptyset$ and after some reordering, the matrix associated to $T$ will be given by
\[\begin{pmatrix}
    M_0 & 0 & 0 & \cdots \cr
    * & * & * & \cdots \cr
    * & * & * & \cdots \cr
    \vdots & \vdots & \vdots & \ddots  
\end{pmatrix}\]
where $M_0$ is a square matrix of size $l$.

We know consider the set $C_{n_1}$ where $n_1 = \min\{k \geqslant 1 \ ; \ k \notin \{0, i_{1,k_1}, \ldots, i_{l-1,k_{l-1}} \}\}$ and we fix $u_{n_1} = e_{n_1}$.
We can then restart the above process with the elements in the set $C_{n_1}\backslash\{0, i_{1,k_1}, \ldots, i_{l-1,k_{l-1}} \}$. By repeating this process, each vector $u_n$ will be well-defined and will satisfy $\operatorname{val}(u_n) = n$ so that the linear operator  $P$ given by $P(e_l) = u_l$ for all $l \geqslant 0$ is invertible by Proposition \ref{prop invertible}. Moreover, by permuting the elements of the basis (if necessary), we get that 
$T$ is conjugated to an operator whose matrix is given by
$$
T = \begin{pmatrix}
     M_0 & 0 & 0 & 0 &\cdots \cr
     * & M_1 & 0 & 0 &\cdots \cr
     * & * & M_2 & 0 &\cdots \cr
     \vdots & \vdots & \vdots & \vdots 
\end{pmatrix} 
$$
where each $M_i$ is a square matrix.
Finally we know, by using the Jordan normal form of each matrix $M_l$, that $T$ is then conjugate to a lower triangular matrix. 
We denote this lower triangular matrix by 
$$
\begin{pmatrix}
    \lambda_0 & 0 & 0 & \cdots \cr
    \alpha_{1,0} & \lambda_1 & 0 & \cdots \cr
    \alpha_{2,0} & \alpha_{2,1} &\lambda_2 & \cdots \cr
    \vdots & \vdots & \vdots & \ddots  
\end{pmatrix}.\\
$$
\text{}\\
\text{}\\
We have now two possibilities: either there exists $\lambda \in \mathbb{C}$ such that $\lambda$ appears a cofinite number of times on the diagonal or this is not the case.
\begin{itemize}
\item If some $\lambda \in \mathbb{C}$ appears cofinitely often on the diagonal then we will have to consider $T -\lambda\operatorname{Id}$ and
turn our attention to the lower diagonal of the matrix of $T-\lambda\operatorname{Id}$.
If there is an infinite number of zeros on this lower diagonal, we will show that $T$ possesses a closed invariant half-space.
However, if there is only finitely many zeros on the lower diagonal, we will show that $T$ is conjugate to $F + R + \lambda \operatorname{Id}$ for some finite rank operator $R$.
\item If there exists no $\lambda \in \mathbb{C}$ that appears cofinitely often on the diagonal, then we will show that $T$ possesses a closed invariant half-space. 
\end{itemize}

\subsubsection{There exists $\lambda \in \mathbb{C}$, such that $\lambda$ appears cofinitely often on the diagonal}

Even if it means considering $T - \lambda \operatorname{Id}$, we can always assume that there is a cofinite number of zeros on the diagonal.
The important information will be on the lower diagonal and we start with the case where there is a cofinite number of nonzero elements on the lower diagonal.

\begin{prop} \label{conj à F+R+lambda}
Let $\lambda\in \mathbb{C}$.
If there exists $k \geqslant 0$ such that for all $l \geqslant k$, we have $\lambda_l = \lambda$ and $\alpha_{l+1,l} \neq 0$ then $T$ is conjugate to $F + R + \lambda \operatorname{Id}$ for $R$ a finite rank operator and for $\lambda \in \mathbb{C}$. 
\end{prop}

\begin{proof}
Let $k \geqslant 0$ such that for all $l \geqslant k$, we have $\lambda_l = \lambda$ and $\alpha_{l+1,l} \neq 0$, then we have that $T = \lambda \operatorname{Id} + R + S$ where $R$ is a finite rank operator and where $S$ is defined by $Sx = 0$ if $x \in \operatorname{span}\{e_0,\ldots,e_{k-1}\}$ and $Sx = (T-\lambda \operatorname{Id})x$ if $x \in \operatorname{\overline{span}}\{e_n \ ; \ n \geqslant k\}$. We remark that for every nonzero vector $x\in \operatorname{\overline{span}}\{e_n \ ; \ n \geqslant k\}$, we have
\[\text{val}(Sx)=\text{val}(x)+1.\]
Therefore, if we let $u_k=e_k$ and if we let for all $l \geqslant k$, $u_{l+1}=Su_l$ then for all $l \geqslant k$, we have  $\operatorname{val}(u_l) = l$.

We now consider the linear application $P$ defined for all $l \geqslant k$ by $P(e_l) = u_l$ and for all $l < k$ by $P(e_l) =e_l$.
By Proposition \ref{prop invertible}, $P$ is invertible and by using this operator, we have that $S$ is conjugate to the operator whose matrix is given by
\[\left(\vcenter{\hbox{
    \begin{tikzpicture}[x=0.5cm, y=-0.4cm]
        \draw[opacity=0] (0.5, 0.5) rectangle (6.5, 6.5);
        \node[font=\LARGE] at (1.25, 1.5) {$0$};
        \node[font=\LARGE] at (4.5, 1.5) {$0$};
        \node[font=\LARGE] at (1.25, 4.5) {$0$};
        \draw (2.5, 6.5) -- (2.5,0); 
        \draw (0,2.5) -- (6.5, 2.5);
        \node at (3,3) {$0$};
        \node at (3,4) {$1$};
        \node at (4,4) {$0$};
        \node at (4, 5) {$1$};
        \draw[dots] (4.5, 5.5) -- (5.5, 6.5);
        \node at (5,5) {$0$};
        \draw[dots] (5.5, 5.5) -- (6.5, 6.5);
    \end{tikzpicture}
}}\right).\]
Moreover, we have 
$$
P^{-1}(T-\lambda \operatorname{Id})P=P^{-1}(R +S)P=R'+P^{-1} SP
$$
where $R'$ is a finite rank operator and thus by Proposition~\ref{rang fini,colonnes}, the operator
 $T-\lambda \operatorname{Id}$ is conjugate to
\[\left(\vcenter{\hbox{
    \begin{tikzpicture}[x=0.5cm, y=-0.4cm]
        \draw[opacity=0] (-0.5,-0.5) rectangle (4.5,4.5);
        \draw (1.5,-0.5) -- (1.5, 5.5);
        \draw (-0.5,1.5) -- (5.5, 1.5);
        \node[font=\LARGE] at (0.5,3.5) {$*$};
        \node[font=\LARGE] at (0.5,0.5) {$*$};
        \node at (2, 2) {$0$};
        \node at (3, 3) {$0$};
        \node at (2, 3) {$1$};
        \node at (3, 4) {$1$};
        \draw[dots] (3.5, 3.5) -- (5, 5);
        \draw[dots] (3.5, 4.5) -- (4.5, 5.5);
        \node[font=\LARGE] at (3.5,0.5) {$0$};
    \end{tikzpicture}
}}\right).\]
Such a matrix is equal to $F + R''$ where $R''$ is a finite rank operator and we get the desired result.
\end{proof}


It remains now to study the case where there is an infinite number of zeros on the lower diagonal.
In this case, we will show that $T$ possesses a closed invariant half-space.

\begin{prop}\label{cupkernel}
Let $\lambda\in \mathbb{C}$. If there exists $k \geqslant 0$ such that for all $l \geqslant k$, we have $\lambda_l = \lambda$ and if there are infinitely many $l \geqslant 0$ such that $\alpha_{l+1,l} = 0$ then $T$ possesses a closed invariant half-space. 
\end{prop}

\begin{proof}
Without loss of generality, we assume that for all $l \geqslant 0$, $\lambda_l = 0$ and thus that for all nonzero $x \in \omega$, $\operatorname{val}(Tx) > \operatorname{val}(x)$.

We first remark that if the set $\bigcup_{n \geqslant 1} \ker (T^n)$ contains elements with arbitrarily big valuation, then $T$ possesses a closed invariant half-space. Indeed, if $x$ is a nonzero element of valuation strictly greater than $j$ in $\operatorname{Ker} T^n$ for some $n \geqslant 2$, we can consider the first $k \leqslant n$ such that $x \in \operatorname{Ker}(T^k)$ and we have that $T^{k-1} x$ is still a nonzero element of valuation strictly greater than $j$ in $\operatorname{Ker}T$ since $T$ is given by a lower triangular matrix. We deduce that we can extract a closed invariant half-space from the kernel of $T$.

We will thus now assume that there exists $N$ such that for all nonzero vector $x$ in $\omega$, if $\text{val}(x) \geqslant N$, then $T^jx \ne 0$ for every $j \geqslant 0$. 
Let $v_n=\text{val}(T^ne_N)$. 
It follows that $(v_n)_{n \geqslant 0}$ is an strictly increasing sequence satisfying $v_0=N$. 
Moreover, if $\alpha_{l+1,l}=0$ and $v_n=l$ then $v_{n+1} \geqslant v_n+2$. In other words, if $\alpha_{l+1,l}=0$ and $l\in \{v_n : n\geqslant 0\}$ then $l+1\notin \{v_n : n\geqslant 0\}$.

We consider the closed invariant subspace $M = \overline{\operatorname{span}}\{T^n e_N \ ; \ n \geqslant 0\}$.
We have that $M$ is an infinite-dimensional subspace and we can also deduce the infinite codimension of $M$ thanks to Lemma \ref{fait}.
Indeed, we have for all $k \geqslant N$ that
\begin{align*}
\operatorname{dim}(P_{\{0,\ldots,k\}}(M)) &= \operatorname{dim}(\operatorname{span}\{T^n e_N \ ; \ \operatorname{val}(T^n e_N) \leqslant k\})\\
&\leqslant |\{n:N \leqslant v_n \leqslant k\}|\\
&\leqslant k+1-N -\frac{|\{N \leqslant l< k: \alpha_{l+1,l}=0\}|}{2}
\end{align*}
and thus 
$$
k+1 - \operatorname{dim}(P_{\{0,\ldots,k\}}(M)) \geqslant N+ \frac{|\{N \leqslant l< k: \alpha_{l+1,l}=0\}|}{2}\underset{k \rightarrow + \infty}{\longrightarrow} +\infty.
$$

\end{proof}

We now have finished to study the case where there exists a $\lambda$ that appears a cofinite number of times on the diagonal.

\subsubsection{There exists no $\lambda \in \mathbb{C}$, such that $\lambda$ appears cofinitely often on the diagonal}

In this case, we show that the operator $T$ possesses a closed invariant half-space.

\begin{prop}
If there exists no $\lambda \in \mathbb{C}$ that appears cofinitely often on the diagonal then $T$ possesses a closed invariant half-space.
\end{prop}

\begin{proof}
We divide the proof into two cases.\\

\textbf{First case} : There exists some $\lambda$ such that $\lambda$ appears infinitely often on the diagonal.\\

Even if it means considering $T - \lambda \operatorname{Id}$, we can assume that there are infinitely many zeros on the diagonal but also infinitely many nonzero elements on the diagonal.

As in the proof of Proposition~\ref{cupkernel}, we first remark that if the set $\bigcup_{n\geqslant 1} \ker (T^n)$ contains elements with arbitrarily big valuation, then $T$ possesses a closed invariant half-space.

We will thus now assume that there exists $N$ such that for all nonzero $x$ in $\omega$, if $\text{val}(x)\geqslant N$, then $T^jx\ne 0$ for every $j\geqslant 0$. 
Moreover, even if it means considering $T$ on $\overline{\text{span}}\{e_k:k \geqslant N\}$, we can assume that $N=0$ and we can even assume that $\lambda_0=0$.

The idea is to build a nonzero vector $x \in \omega$, a non-decreasing sequence $(\nu_j)_{j \geqslant 0}$ and an increasing sequence $(n_j)_{j \geqslant 0}$ such that for all $j \geqslant 0$: 
\begin{itemize}
    \item[$\bullet$] $\nu_j \leqslant |\{l \in \{0,\ldots,j\} \ ; \ \lambda_{n_l} = 0\}|$;
    \item[$\bullet$] for all $n \geqslant \nu_j$, $\operatorname{val}(T^n x) > n_j$.
\end{itemize}
Let $j \geqslant 0$. 
The second condition will be checked as soon as we have that $\operatorname{val}(T^{\nu_j} x) > n_j$ since $T$ is given by a lower triangular matrix. Such a sequence $(T^n x)_{n \geqslant 0}$ will then allow us to build a closed invariant half-space.

We have that 
$$ 
Te_0=\begin{pmatrix}
    0 & 0 & 0 & \cdots \cr
    \alpha_{1,0} & \lambda_1 & 0 & \cdots \cr
    \alpha_{2,0} & \alpha_{2,1} &\lambda_2 & \cdots \cr
    \vdots & \vdots & \vdots & \ddots  
\end{pmatrix} 
\cdot e_0
= \begin{pmatrix}
    0 \cr
    \alpha_{1,0} \cr
    \alpha_{2,0} \cr
    \vdots 
\end{pmatrix}
$$
and thus $P_0 T(e_0)$ = 0.
Knowing that $e_0$ does not belong to $\operatorname{Ker}T$, there exists $n \geqslant 1$ such that $\alpha_{n,0} \neq 0$.
Consider the first $n_1 \geqslant 1$ such that $\alpha_{n_1,0} \neq 0$. 
We fix $x_{0}=1$ and $x_1= \ldots = x_{n_1-1} = 0$.
We also fix $\nu_0 = 1$, $n_0 = 0$ so that $\nu_0 \leqslant |\{l \in \{0\} \ ; \ \lambda_{n_l} = 0\}|$ and $\operatorname{val}(T^{\nu_0}x) > n_{0}$.
We now want to fix $x_{n_1}$ and $\nu_1$ such that $P_{\{0, \ldots, n_1\}}T^{\nu_1}(x_0, \ldots, x_{n_1}, 0, \ldots) = 0$ with $\nu_{1} \leqslant |\{l \in \{0,1\} \ ; \ \lambda_{n_l} = 0\}|$.
We have that
$$
\begin{pmatrix}
    0 & 0 & 0 & \cdots & \cdots \cr
    0 & \lambda_1 & 0 & \cdots & \cdots \cr
    \vdots & \vdots & \ddots & \cdots \cr
    0 & \cdots & \cdots & \lambda_{n_1-1} & \cdots \cr
    \alpha_{n_1,0} & \cdots & \cdots & \cdots & \lambda_{n_1}
\end{pmatrix}
\cdot \begin{pmatrix}
    x_{0} \cr
    x_1 \cr
    \vdots \cr
    x_{n_1-1} \cr
    x_{n_1} 
\end{pmatrix}
=
\begin{pmatrix}
    0 \cr
    0 \cr
    \vdots \cr
    0 \cr
    \alpha_{n_1,0} + \lambda_{n_1} x_{n_1}  
\end{pmatrix}
.
$$
If $\lambda_{n_1} \neq 0$, we fix $x_{n_1} = -\frac{\alpha_{n_1,0}}{\lambda_{n_1}}$ so that $P_{\{0, \ldots, n_1\}}T^{\nu_0}(x) = 0$ and we can then consider $\nu_1 = 1$. On the other hand, if $\lambda_{n_1} = 0$ then we fix 
$x_{n_1}=0$ and by applying $T$ a second time, we have that 
$$
\begin{pmatrix}
    0 & 0 & 0 & \cdots & \cdots \cr
    0 & \lambda_1 & 0 & \cdots & \cdots \cr
    \vdots & \vdots & \ddots & \cdots \cr
    0 & \cdots & \cdots & \lambda_{n_1-1} & \cdots \cr
    \alpha_{n_1,0} & \cdots & \cdots & \cdots & \lambda_{n_1}
\end{pmatrix}
\cdot \begin{pmatrix}
    0 \cr
    0 \cr
    \vdots \cr
    0 \cr
    \alpha_{n_1,1}  
\end{pmatrix}
= \begin{pmatrix}
    0 \cr
    0 \cr
    \vdots \cr
    0 \cr
    0  
\end{pmatrix}
.
$$
In this case, we can then consider $\nu_1 = 2$.

Suppose that we have fixed the first $n_j$ coordinates of $x$ and suppose that we have $\nu_{j} \leqslant |\{l \in \{0,\ldots,j\} \ ; \ \lambda_{n_l} = 0\}|$ and $\operatorname{val}(T^{\nu_j}x) > n_j$.

Given that $(x_0, \ldots, x_{n_j}, 0, \ldots)$ cannot belong to the kernel of $T^{\nu_j}$, the vector $T^{\nu_j}(x_0, \ldots, x_{n_j}, 0, \ldots)$ has a nonzero coordinate.
We consider the first $n_{j+1} \geqslant 0$ such that $P_{n_{j+1}}T^{\nu_j}(x_0, \ldots, x_{n_j}, 0, \ldots) \neq 0$. Since we know that the first $n_j$ coordinates of $T^{\nu_j}x$ are zero, we deduce that $n_{j+1}>n_j$. 
We then fix $x_{n_j + 1} = \ldots = x_{n_{j+1} -1} = 0$ and we want to fix $x_{n_{j+1}}$ and $\nu_{j+1}$ such that $P_{\{0, \ldots, n_{j+1}\}}T^{\nu_{j+1}} (x_0, \ldots,\\ x_{n_{j+1}}, 0, \ldots) = 0$ with $\nu_{j+1} \leqslant |\{l \in \{0,\ldots,j+1\} \ ; \ \lambda_{n_l} = 0\}|$.

As previously, we get that the $(n_{j+1})$th coordinate of $T^{\nu_j} x$ is given by
$$
P_{n_{j+1}}T^{\nu_j}(P_{\{0,\ldots,n_j\}} x) + \lambda_{n_{j+1}}^{\nu_j} x_{n_{j+1}}.
$$
If $\lambda_{n_{j+1}} \neq 0$, then we fix $x_{n_{j+1}}$ such that $P_{n_{j+1}}T^{\nu_j}(P_{\{0,\ldots,n_j\}} x) + \lambda_{n_{j+1}}^{\nu_j} x_{n_{j+1}} = 0$ and we consider $\nu_{j+1} = \nu_j$.
On the other hand, if $\lambda_{n_{j+1}} = 0$ then we fix $x_{n_{j+1}}=0$ and by applying $T$ one more time, we get that $\text{val}(T^{\nu_j +1}x)>n_{j+1}$ and we can thus consider $\nu_{j+1} = \nu_j + 1$.

It remains to remark that the sequence $(\nu_k)_{k \geqslant 0}$ cannot be ultimately constant because if it was the case, the vector $x$ will belong to the kernel of $T^{\nu}$ (where $\nu$ is the limit of the sequence $(\nu_k)_{k \geqslant 0}$) which is a contradiction. Therefore if we consider the closed invariant subspace $M = \overline{\operatorname{span}}\{T^n x \ ; \ n \geqslant 0\}$ then since $(\nu_k)_{k \geqslant 0}$ is not ultimately constant, $M$ is an infinitely dimensional subspace.
Moreover, we deduce from Lemma \ref{fait} that $M$ has an infinite codimension.
Indeed, we have $|\{n \ ; \ \operatorname{val}(T^n x) \leqslant n_j\}| \leqslant \nu_{j}$.
Then, 
$$
\operatorname{dim}(P_{\{0,\ldots,n_j\}}(M)) = \operatorname{dim}(\operatorname{span}\{T^n x \ ; \ \operatorname{val}(T^nx) \leqslant n_j\}) \leqslant \nu_j.
$$
Moreover, we know that $\nu_j \leqslant |\{l \in \{0,\ldots,j\} \ ; \ \lambda_{n_l} = 0\}|$. 
Then, 
$$
n_j + 1 - \operatorname{dim}(P_{\{0,\ldots,n_j\}}(M)) \geqslant |\{l \in \{0, \ldots,n_j\} \ ; \ \lambda_{l} \neq 0\}| \underset{j \rightarrow + \infty}{\longrightarrow} +\infty
$$
because there is an infinite number of $l$ such that $\lambda_{l} \neq 0$.\\

\textbf{Second case :} Each $\lambda$ on the diagonal only appears finitely many times.\\

In this case, we deduce that there are infinitely many $\lambda$ such that $\lambda$ appears finitely often on the diagonal. We can thus find a sequence $(\lambda_i)_{i \geqslant 0}$ of elements appearing on the diagonal and a sequence $(k_i)_{i \geqslant 0}$ such that for all $i \geqslant 0$, $k_i$ is the index of the last appearance of $\lambda_i$ on the diagonal and such that the sequence $(k_i)_{i \geqslant 0}$ is strictly increasing.

Let $\lambda \in \{\lambda_i : i \geqslant 0\}$ and $k$ the last appearance of $\lambda$ on the diagonal. 
We have that 
\begin{align*}
&Te_k =
\begin{pmatrix}
    \lambda_0 & 0 & 0 & \cdots \cr
    \vdots & \ddots & 0 & \cdots \cr
     & \cdots & \lambda_{k} & \cdots \cr
     & \cdots & \alpha_{k+1,k} & \cdots \cr
    \vdots & \vdots & \vdots & \ddots 
\end{pmatrix}
\cdot \begin{pmatrix}
    0 \cr
    \vdots \cr
    0 \cr
    1 \cr
    0 \cr
    \vdots \cr
    0
\end{pmatrix} \\
&= (0,\ldots,\underset{k-1}{0}, \lambda_k, \alpha_{k+1,k}, \ldots). 
\end{align*}
So, 
$$
(T- \lambda \operatorname{Id} )e_k =  (0,\ldots,\underset{k}{0},\alpha_{k+1,k},\alpha_{k+2,k},\ldots).
$$
For all $l \geqslant k+1$, we have that $\lambda_l - \lambda \neq 0$ and the operator $T-\lambda \operatorname{Id}$ is thus invertible from $\overline{\text{span}}\{e_n:n\geqslant k+1\}$ to $\overline{\text{span}}\{e_n:n\geqslant k+1\}$. In other words, there exists $x \in \overline{\text{span}}\{e_n:n\geqslant k+1\}$ such that $(T-\lambda \operatorname{Id} ) x= -(0,\ldots,\underset{k}{0},\alpha_{k+1,k},\alpha_{k+2,k},\ldots)$ and thus such that $e_k+x\in \ker(T- \lambda \operatorname{Id} )$. In conclusion, for each $i \geqslant 1$, the operator $T$ possesses an eigenvector with valuation equal to $k_i$ and we can thus easily construct a closed invariant half-space for $T$ by using them.
\end{proof}

\section{Consequences}\label{SecFinal}

\subsection{Remarks on the theorem}

We will now finish the proof of Theorem \ref{mainthm} by showing that thanks to the previous characterization, we have that an operator $T \in L(\omega)$ has a closed invariant half-space if and only if $T$ has a closed almost-invariant half-space.

\begin{prop}
Let $T \in L(\omega)$.
Then $T$ has a closed invariant half-space if and only if $T$ has a closed almost-invariant half-space.
\end{prop}

\begin{proof}
Let $T \in L(\omega)$ be an operator without a closed invariant half-space and $R \in L(\omega)$ be a finite rank operator.
By the equivalence between $2.$ and $3.$ in Theorem~\ref{mainthm}, we know that $T$ is conjugate to $S + R' + \lambda \operatorname{Id}$ for some finite rank operator $R'$ and some $\lambda \in \mathbb{C}$ where $S = B$ or $F$.
It follows that $T+R$ is conjugate to $S + R'' + \lambda \operatorname{Id}$ where $R''$ is also a finite rank operator.
Then, again by the equivalence between $2.$ and $3.$ in Theorem \ref{mainthm}, we can deduce that $T+R$ does not possess a closed invariant half-space and thus $T$ does not possess a closed almost-invariant half-space.
\end{proof}

This result allows us to get the equivalence between $1.$ and $2.$ in Theorem \ref{mainthm} and to conclude the proof of Theorem~\ref{mainthm}.\\


In view of Remark~\ref{no Id for B}, one might wonder if it was possible to state Theorem~\ref{mainthm} more succinctly. However, excepted the fact mentioned in Remark~\ref{no Id for B}, we cannot simplify the statement of Theorem~\ref{mainthm}. Indeed we have that 
    \begin{itemize}
        \item $F+ \lambda \operatorname{Id}$ with $\lambda \neq 0$ is not conjugate to $F + R$ with $R \in L(\omega)$ a finite rank operator.
        As a consequence, $F+ \lambda \operatorname{Id}$ is not conjugate to $F + \mu \operatorname{Id}$ if $\lambda \neq \mu$.
    \end{itemize}
Indeed, first of all, we know by Proposition \ref{rang fini,colonnes} that $R$ has a finite number of nonzero columns.
Assume that $F+ \lambda \operatorname{Id} = P^{-1}(F+R) P$ for $P \in L(\omega)$ an invertible operator.
We can then consider $n_0$ such that $Re_n = 0$ for all $n \geqslant n_0$ and we let $u_n=P^{-1}(e_n)$ for every $n \geqslant n_0$.
We know by continuity of $P^{-1}$ that $(\operatorname{val}(u_n))_{n \geqslant n_0}$ tends to infinity. 
We can then find $n \geqslant n_0$ such that $\operatorname{val}(u_n) < \operatorname{val}(u_{n+1})$ so that
$$
\operatorname{val}(P^{-1}(F+R) P(u_n)) = \operatorname{val}(P^{-1}(Fe_n)) = \operatorname{val}(P^{-1}(e_{n+1})) = \operatorname{val}(u_{n+1}) 
$$
and
$$
\operatorname{val}((F+ \lambda \operatorname{Id})u_n) = \operatorname{val}(F u_n + \lambda u_n)=\operatorname{val}(u_n).
$$
It is then impossible to have that $P^{-1}(F+R) P(u_n) = (F+ \lambda \operatorname{Id})u_n$.\\

Moreover, we have that

\begin{itemize}
    \item $B$ is not conjugate to $F + \lambda \operatorname{Id}+ R$ with $\lambda \in \mathbb{C}$ and $R\in L(\omega)$ a finite rank operator.
\end{itemize}

Assume that $B = P^{-1}(F+\lambda \operatorname{Id}+ R) P$ for $P \in L(\omega)$ an invertible operator.
We can then as previously consider $n_0$ such that $Re_n = 0$ for all $n \geqslant n_0$. Let $u_n=P^{-1}(e_n)$ for every $n \geqslant n_0$. Then by continuity of $P^{-1}$, we know that $(\operatorname{val}(u_n))_{n \geqslant n_0}$ tends to infinity and  
we can thus find $n \geqslant n_0$ such that $1 \leqslant \operatorname{val}(u_n) < \operatorname{val}(u_{n+1})$. It follows that
$$
\operatorname{val}(P^{-1}(F+\lambda \operatorname{Id}+R) P(u_n)) = \operatorname{val}(P^{-1}(Fe_n+\lambda e_n)) = \operatorname{val}(u_{n+1}+\lambda u_n) \geqslant \operatorname{val}(u_{n}) 
$$
and
$$
\operatorname{val}(B u_n) = \operatorname{val}(u_n)-1.
$$
It is then impossible to have that $P^{-1}(F+\lambda \operatorname{Id}+R) P = B$.\\




\subsection{Some particular examples}
We will see some particular examples of operators defined on $\omega$ to illustrate the different cases in the proof of Theorem~\ref{mainthm}.

First of all, we will look at the operators $T$ defined by $\sum_{k \geqslant 0} a_k F^k$ with $(a_k)_{k \geqslant 0} \subseteq \mathbb{C}$.
Then, in this case, we have that the matrix of $T$ is given by 
$$
\begin{pmatrix}
    a_0 & 0 & 0 & \cdots \cr
    a_1 & a_0 & 0 & \cdots \cr
    a_2 & a_1 & a_0 & \cdots \cr
    \vdots & \vdots & \ddots & \ddots  
\end{pmatrix}. 
$$
The existence of closed invariant half-spaces for this operator is equivalent to the existence of closed invariant half-spaces for $T - a_0 \operatorname{Id}$.
Then, we have two cases.
If $a_1 = 0$, then we are in the case where there is a cofinite number of $0$ on the diagonal and there is an infinite number of $0$ on the lower diagonal, then we can conclude by Proposition \ref{cupkernel} that $T - a_0 \operatorname{Id}$, and thus also $T$ has a closed invariant half-space. On the other hand, if $a_1 \ne 0$, we are in the case where there is a cofinite number of $0$ on the diagonal and there is a cofinite number of nonzero elements on the lower diagonal, then we can conclude by Proposition \ref{conj à F+R+lambda} that $T$ is conjugate to $F + a_0 \operatorname{Id}$.

We can also consider the operator $T$ defined by $\sum_{k \geqslant 0} b_k B^k$ with $(b_k)_{k \geqslant 0} \subseteq \mathbb{C}$.
Then, in this case, we have that the matrix of $T$ is given by 
$$
\begin{pmatrix}
    b_0 & b_1 & b_2 & \cdots \cr
    0 & b_0 & b_1 & \cdots \cr
    0 & 0 & b_0 & \cdots \cr
    \vdots & \vdots & \vdots & \ddots  
\end{pmatrix}. 
$$
We can first remark that to have the continuity of the operator $T$, we must have that there is only a finite number of nonzero $b_k$.
Otherwise, we would have that there exists an infinite family of indices $m \geqslant 1$ such that $P_0 T(e_m) \neq 0$.
We have already seen in the proof of Proposition \ref{rang fini,colonnes} that this is impossible. We can in fact deduce from the following result that such an operator $T$ does not possess a closed invariant half-space if and only if $b_1\ne 0$ and $b_k=0$ for every $k\geqslant 2$.

\begin{thm}
Let $T=\sum_{k\geqslant 0}a_k F^k+\sum_{k=1}^d b_k B^k$.
Then $T$ does not possess a closed invariant half-space if and only if one of the following conditions is satisfied:
\begin{itemize}
\item $a_1\ne 0$ and $b_k=0$ for every $1 \leqslant k \leqslant d$;
\item $b_1\ne 0$ and $b_k=0$ for every $2 \leqslant k \leqslant d$.
\end{itemize}
\end{thm}
\begin{proof}
First of all, if for all $1 \leqslant k \leqslant d$, $b_k = 0$, then $T = \sum_{k\geqslant 0}a_k F^k$.
As a consequence, we already know that in this case, $T$ does not possess a closed invariant half-space if and only if $a_1 \neq 0$.
 
We can now suppose that there exists at least one $1 \leqslant k \leqslant d$ such that $b_k \neq 0$.
If $b_1\ne 0$ and $b_k=0$ for every $2 \leqslant k \leqslant d$, we have that $T = b_1 B + \sum_{k\geqslant 0}a_k F^k$.
In this case, we have that for all $l \geqslant 1$, $A_l=\{l\}$ and thus $|A_l| = 1$. 
By Proposition~\ref{|A_l| = 1}, we can conclude that $T$ does not possess a closed invariant half-space.

To conclude the proof, it remains to show that if there exists $2 \leqslant m \leqslant d$, such that $b_m \neq 0$, then $T$ possesses a closed invariant half-space.
We consider the largest $m$ such that $b_m \neq 0$ and without loss of generality, we can then suppose that $m=d$.
We can first notice that for all $l\geqslant 1$, if $n > ld$ then $P_0 T^l e_n = 0$ and for all $l \geqslant 1$, $ld \in A_l$ and $P_0 T^l e_{ld} = b_d^l$.

We now consider different cases. 
We first assume that there exists $k < d$ such that $b_k \neq 0$. 
In this case, it follows that $T = b_d B^d + b_k B^k + \sum_{j=1}^{k-1}b_j B^j + \sum_{j\geqslant 0}a_{j} F^{j}$ and that for all $l \geqslant 1$, $(l-1)d + k \in A_l$ and $P_0T^l e_{(l-1)d+k} = lb_d^{l-1}b_k$.
We can then conclude that $|A_l| \geqslant 2$ for all $l \geqslant 1$ and thus, by Proposition \ref{|A_l| >= 2}, $T$ has a closed invariant half-space.

Assume now that we have $b_k=0$ for all $1\leqslant k<d$ and let us divide this last case into two.
\begin{itemize}
\item Assume that for all $k\geqslant 0$ if $a_k \neq 0$ then $k \in d \mathbb{Z}_{+}$. Then, we can write $T = b_d B^d + \sum_{k\geqslant 0}a_{kd} F^{kd}$ and it follows that if $j \notin d \mathbb{Z}_+$, then $j \notin A$.
Since $A$ is not cofinite, we can conclude thanks to Proposition \ref{notcofinite} that $T$ possesses a closed invariant half-space.

\item Assume that there exists $k \notin d \mathbb{Z}_{+}$ such that $a_k \neq 0$.
Consider the smallest $k \notin d\mathbb{Z}_+$ such that $a_k \neq 0$.
Then, we can write $T = b_d B^d + \sum_{j\geqslant k}a_{j} F^{j} + \sum_{j : jd < k} a_{jd} F^{jd}$.

Let $l_0$ be the smallest integer $l$ such that $(l-1)d \geqslant k$.
We can show that for all $l \geqslant l_0$, $(l-1)d -k \in A_l$. Indeed, we can compute that for every $l\geqslant l_0$, $P_0 T^le_{(l-1)d-k}=(l-l_0+1)a_kb_d^{l-1}\ne 0$ and that for every $l\geqslant l_0$, if $l'< l$ then 
\begin{align*}
P_0T^{l'}e_{(l-1)d-k} &= b_d P_0 T^{l'-1} B^d e_{(l-1)d-k} + a_k P_0 T^{l' -1} e_{(l-1)d} \\
&+ \sum_{j \geqslant k+1} a_j P_0 T^{l' -1} e_{(l-1)d+j-k}\\
&+ \sum_{j : jd < k} a_{jd} P_0 T^{l' -1} e_{(l+j-1)d-k} = 0.
\end{align*}

The second and third terms are clearly equal to zero since $l'<l$. The last term is equal to $0$ because if we want $P_0 T^n e_{(l+j-1)d-k}$ to be nonzero, we have to rely (at least once) on some $F^{k'}$ where $k'$ is not a multiple of $d$ and thus on some $k'\geqslant k$. However, since $(l+j-1)d-k+k'$ is then bigger than $(l-1)d$, it cannot be done by applying $T$ only $l'-2$ times more. Finally the first term is equal to zero for $l=l_0$ by definition of $l_0$ and it will then be equal to zero for every $l\geqslant l_0$ by induction.

Therefore we can conclude that for all $l\geqslant l_0$, $ld$ and  $(l-1)d-k$ belong to $A_l$. It follows that $|A_l| \geqslant 2$ for all $l \geqslant l_0$ and by Proposition \ref{|A_l| >= 2} that $T$ possesses a closed invariant half-space.
\end{itemize}
\end{proof}

\subsection{Case of $X \oplus \omega$ with $X$ a Fréchet space with a continuous norm}

Now that we have thoroughly studied the problem of the almost-invariant half-space for operators defined on $\omega$, we will turn our attention to this problem for operators defined on $X \oplus \omega$ with $X$ a Fréchet space with a continuous norm. 
First of all, let us recall that on $X\oplus \omega$ where $X$ is a Fréchet space with a continuous norm, every bounded operator possesses a non-trivial closed invariant subspace \cite{menet2018invariant}. 
If $X$ is a Fréchet space with a continuous norm, there is thus no difference between $\omega$ and $X\oplus \omega$ for the Invariant Subspace Problem. As shown by Menet, the difference for the Invariant Subspace Problem can only appear if $X$ is a Fréchet space without continuous norm \cite{menet2021invariant}.


The situation is different for the Almost-Invariant Subspace Problem.
Indeed, if $X$ is a Banach space of finite dimension, we have that $X \oplus \omega$ is isomorphic to $\omega$ and we get the same result than for $\omega$. But, if $X$ is an infinite-dimensional Banach space or more generally an infinite-dimensional Fréchet space with a continuous norm then every operator defined on $X \oplus \omega$ has a closed invariant half-space as shown below. 

\begin{thm} \label{X + omega}
Let $X$ be an infinite-dimensional Fréchet space with a continuous norm. 
Then each operator defined on $X \oplus \omega$ possesses a closed almost-invariant half-space.
\end{thm}

\begin{proof}
Let $T$ be an operator defined on $X \oplus \omega$.
We are going to show that $\{0\} + \omega$ is a closed almost-invariant half-space for $T$.
Let us already note that $\{0\} + \omega$ is a half-space since $X$ is infinite-dimensional.
It remains to show that this half space is almost-invariant for $T$.

Let $(\|\cdot\|_j)_{j \geqslant 0}$ be an increasing sequence of norms defining the topology of $X$. 
Let $x + v \in X \oplus \omega$ with $x\in X$ and $v\in \omega$. 
For every $j \geqslant 0$, we fix $q_j(x+v) = \|x\|_j + p_j(v)$.
We then have that the sequence $(q_j)_{j \geqslant 0}$ defines the topology of $X \oplus \omega$ and we can write
$$
T(x+v) = T_{XX} x + T_{\omega X} v + T_{X \omega} x + T_{\omega \omega} v 
$$
where $T_{XX} : X \rightarrow X$, $T_{\omega X} : \omega \rightarrow X$, $T_{X \omega} : X \rightarrow \omega$ and $T_{\omega \omega} : \omega \rightarrow \omega$ are continuous. 

We then remark that $T_{\omega X}$ is a finite rank operator. Indeed, by continuity of $T_{\omega X}$, there exist $C > 0$ and $J \in \mathbb{Z}_+$ such that for all $v \in \omega$, 
$$
\|T_{\omega X} v\|_1 \leqslant C p_J(v), \quad \text{i.e.} \quad \|T_{\omega X} v\|_1 \leqslant C \sup_{0 \leqslant i \leqslant J} |v_i|.
$$ 
This implies that $\operatorname{Im}(T_{\omega X}) \subseteq \operatorname{span}\{T_{\omega X} e_0,\ldots,T_{\omega X} e_{J}\}$ and thus that $T_{\omega X}$ is a finite rank operator. The operator $R:X\oplus \omega\to X\oplus \omega$ defined by $R(x+v)=-T_{\omega X}v$ is then a finite rank operator and for every $0+v \in \{0\}+\omega$, we have that 
$$
(T+R)(0+v) = T_{XX} 0 + T_{\omega X} v + T_{X \omega} 0 + T_{\omega \omega} v- T_{\omega X} v= 0 + T_{\omega \omega} v\in \{0\}+\omega.
$$
The space $\{0\}+\omega$ is thus a closed almost-invariant half-space for $T$.
\end{proof}



\nocite{*}
\printbibliography

\end{document}